\documentclass[11pt,a4paper]{article}

\usepackage[T1]{fontenc}
\usepackage[utf8]{inputenc}
\usepackage{lmodern}
\usepackage{microtype}
\usepackage{geometry}
\usepackage{mathtools,amssymb,amsthm}
\usepackage{enumitem}
\usepackage{xcolor}
\usepackage{hyperref}

\setlist{itemsep=0.25em,topsep=0.5em}
\hypersetup{
  pdftitle={Local classification of chsc K\"ahler metrics with cone singularities},
  pdfauthor={Martin de Borbon},
  colorlinks=true,
  linkcolor=blue!45!black,
  citecolor=blue!45!black,
  urlcolor=blue!45!black
}

\numberwithin{equation}{section}

\newtheorem{theorem}{Theorem}[section]
\newtheorem{proposition}[theorem]{Proposition}
\newtheorem{lemma}[theorem]{Lemma}
\newtheorem{corollary}[theorem]{Corollary}
\theoremstyle{definition}
\newtheorem{definition}[theorem]{Definition}
\theoremstyle{remark}
\newtheorem{remark}[theorem]{Remark}

\newcommand{\C}{\mathbb{C}}
\newcommand{\N}{\mathbb{N}}
\newcommand{\R}{\mathbb{R}}
\newcommand{\Z}{\mathbb{Z}}
\newcommand{\cO}{\mathcal{O}}
\newcommand{\End}{\operatorname{End}}
\newcommand{\Res}{\operatorname{Res}}
\newcommand{\tr}{\operatorname{tr}}
\newcommand{\Ric}{\operatorname{Ric}}
\newcommand{\Id}{\operatorname{Id}}
\newcommand{\Span}{\operatorname{span}}
\newcommand{\ii}{\mathrm{i}}

\title{Local classification of chsc K\"ahler metrics\\ with cone singularities}
\author{Martin de Borbon}
\date{}

\begin{document}

\maketitle

\begin{abstract}
Let \(B\subset\C^n\) be a ball centred at the origin and \(D=\{z^1=0\}\).  Let \(g\) be a K\"ahler
metric of constant holomorphic sectional curvature (chsc) on \(B\setminus D\), uniformly
equivalent to the model cone metric of angle \(2\pi\beta\), with \(0<\beta<1\), and
polyhomogeneous along \(D\).  We prove that, in suitable holomorphic coordinates \((w^1, \ldots, w^n)\) centred at the origin, the metric \(g\) is the pullback of the corresponding complex space form by
the map \((w^1,w^2,\dots,w^n)\longmapsto
\bigl((w^1)^\beta,w^2,\dots,w^n\bigr)\).
\end{abstract}



\section{Introduction}\label{sec:intro}

Let \(B\subset\C^n\) be a ball centred at the origin, with holomorphic coordinates
\(z=(z^1,z')\), and put
\[
  D=\{z^1=0\}\subset B .
\]
Throughout, indices \(i,j,k,\ell\) run from \(1\) to \(n\) and indices \(a,b,c\) run
from \(2\) to \(n\).  Fix \(0<\beta<1\) and set
\begin{equation}\label{eq:model}
  g_\beta=|z^1|^{2\beta-2}|dz^1|^2+\sum_{a=2}^n|dz^a|^2 ,
\end{equation}
the product of the flat two-dimensional cone of total angle \(2\pi\beta\) with a
Euclidean factor.  We denote its K\"ahler form by \(\omega_\beta\), and write
\(g_{\mathrm{Euc}}\) and \(\omega_{\mathrm{Euc}}\) for the Euclidean metric and its
K\"ahler form, respectively.

\begin{definition}\label{def:cone}
A K\"ahler metric \(g\) on \(B\setminus D\) has a \emph{cone singularity of angle
\(2\pi\beta\)} along \(D\) if
\begin{equation}\label{eq:edge-comparison}
  C^{-1}g_\beta\le g\le Cg_\beta
\end{equation}
locally along \(D\) for some \(C\ge1\).
\end{definition}

We abbreviate \emph{constant holomorphic sectional curvature} to \emph{chsc}
throughout, and use the curvature convention
\begin{equation}\label{eq:chsc-convention}
  R_{i\bar jk\bar l}
  =\frac{c}{2}\bigl(g_{i\bar j}g_{k\bar l}+g_{i\bar l}g_{k\bar j}\bigr),
\end{equation}
so that \(g\) is chsc with curvature \(c\); contracting \eqref{eq:chsc-convention}
gives \(\Ric(g)=\tfrac{n+1}{2}c\,g\).  For \(\sigma\in\{1,-1\}\) let \(X_\sigma\) be
\(\C P^n\) if \(\sigma=1\) and the unit ball in \(\C^n\) if \(\sigma=-1\), equipped in
affine coordinates \(Z=(Z^1,\dots,Z^n)\) with the K\"ahler form
\begin{equation}\label{eq:space-form}
  \omega_\sigma
  =\sigma\,\ii\,\partial\bar\partial
    \log\Bigl(1+\sigma\sum_{i=1}^n|Z^i|^2\Bigr),
\end{equation}
which is chsc with curvature \(2\sigma\) in the convention
\eqref{eq:chsc-convention}.  We denote the associated K\"ahler metric by
\(g_\sigma\).

The local classification of chsc K\"ahler metrics is classical and goes
back to Bochner \cite[Theorem~6]{Bochner1947}. In this paper we prove the corresponding statement for metrics with cone singularities. Our main result is the following.

\begin{theorem}[Local classification]\label{thm:main}
Let \(g\) be a chsc K\"ahler metric on \(B\setminus D\), with curvature \(c\),
and cone angle \(2\pi\beta\) along \(D\) for some \(0<\beta<1\). Suppose that \(g\) is polyhomogeneous along \(D\) in the sense of Definition~\ref{def:polyhomogeneous}.  After shrinking \(B\) there are holomorphic coordinates \(w=(w^1,\dots,w^n)\) centred at the origin such that the following holds.
\begin{enumerate}[label=\textup{(\roman*)}]
\item If \(c=0\), then
\begin{equation}\label{eq:main-flat}
  g=|w^1|^{2\beta-2}|dw^1|^2+\sum_{a=2}^n|dw^a|^2 .
\end{equation}
\item If \(c\neq0\), then, with \(\sigma=\operatorname{sgn}c\), the K\"ahler form of
\(g\) is
\begin{equation}\label{eq:main-chsc}
  \omega=\frac{2\sigma}{|c|}\,\ii\,\partial\bar\partial
   \log\Bigl(1+\sigma\bigl(|w^1|^{2\beta}+\textstyle\sum_{a=2}^n|w^a|^2\bigr)\Bigr) .
\end{equation}
\end{enumerate}
\end{theorem}

\begin{remark}\label{rem:polyhomogeneity-appendix}
The polyhomogeneity hypothesis is automatic under the other hypotheses of the
theorem.  Indeed, a chsc K\"ahler metric is
K\"ahler--Einstein, and Theorem~\ref{thm:appendix-polyhomogeneity} in
Appendix~\ref{app:polyhomogeneity} proves that every K\"ahler--Einstein metric
satisfying \eqref{eq:edge-comparison} is polyhomogeneous in the sense of
Definition~\ref{def:polyhomogeneous}.  We retain polyhomogeneity as an explicit
hypothesis so that the classification statement and its proof are independent of
this regularity result.
\end{remark}

Both cases say the same thing.  Let
\begin{equation}\label{eq:branched-map}
  F(w)=\bigl((w^1)^\beta,w^2,\dots,w^n\bigr).
\end{equation}
Then \eqref{eq:main-chsc} reads
\(\omega=(2/|c|)F^*\omega_\sigma\), and \eqref{eq:main-flat} is
\(g=F^*g_{\mathrm{Euc}}\) after the constant rescaling
\(w^1\mapsto\beta^{-1/\beta}w^1\).

\subsection*{Outline}

Sections~\ref{sec:log-extension}--\ref{sec:classification} treat the flat case.  When
\(c=0\) the Christoffel symbols are holomorphic off \(D\). Theorem~\ref{thm:log-extension} shows that the Levi--Civita connection of \(g\) extends over \(TB\) with a logarithmic pole along \(D\) with residue of rank one and trace \(\beta-1\).
Theorem~\ref{thm:connection-normal-form} provides normal coordinates for flat torsion-free logarithmic connection whose residue has non-integral trace
\(\lambda\).
Section~\ref{sec:classification} integrates \(\nabla g=0\) in these coordinates and
proves Theorem~\ref{thm:main}(i).
Section~\ref{sec:jet} treats \(c\neq0\).  The flat object is now the Chern connection
\(\mathcal D\) of a Hermitian form \(\mathcal H\) on a bundle \(E\) of first jets,
introduced in \cite{deBorbon2025}.  Using this connection, we prove
Theorem~\ref{thm:main}(ii).

\subsection*{Acknowledgements}

I want to thank Dima Panov, from whom I learnt logarithmic connections. Many of the ideas of this paper developed during our collaboration, especially \cite{deBorbonPanov2021, deBorbonPanov2022}.

\subsection*{Use of artificial intelligence}

The main result of this paper, Theorem~\ref{thm:main}, was announced in
\cite{deBorbon2025}.  The strategy of proof is due to the author and goes back to
that paper, written at a time when the author was not using AI tools in his
research.  ChatGPT was used to produce the present article from an earlier draft.
That draft carried polyhomogeneity as a hypothesis.  The argument of
Appendix~\ref{app:polyhomogeneity}, which shows that this hypothesis is automatic,
was largely developed through conversations with ChatGPT.

\section{Polyhomogeneous K\"ahler metrics}\label{sec:regularity}

Write \(\rho=|z^1|\) and \(\theta=\arg z^1\).  On a simply connected sector of \(B\setminus D\) choose a branch of
\begin{equation}\label{eq:cone-coordinate}
  \zeta=\frac{(z^1)^\beta}{\beta},
\end{equation}
so that \((\zeta,z')\) are holomorphic coordinates in which \(g_\beta\) is Euclidean;
we call these \emph{cone coordinates} and set
\begin{equation}\label{eq:r-rho}
  r=|\zeta|=\rho^\beta/\beta .
\end{equation}
Objects carrying a tilde are computed in the cone coordinates, whose indices we label
as those of \(z\), the index \(1\) referring to \(\zeta\).  Two branches of
\eqref{eq:cone-coordinate} differ by a constant phase factor, so all the
estimates below are independent of the choice.  The variables \((r,\theta,z')\) are
those of \cite{JMR2016}, whose conventions for asymptotics we follow.

\begin{definition}[Polyhomogeneous functions; {\cite[Definition~2.6]{JMR2016}}]
\label{def:phg}
For \(\nu\in\R\), let \(\mathcal A^\nu\) be the space of smooth functions \(u\) on
\(B\setminus D\) such that
\[
  Pu=O(r^\nu)
\]
for every finite product \(P\) of the vector fields
\(r\partial_r\), \(\partial_\theta\), \(\partial_{z^a}\), and
\(\partial_{\bar z^a}\).  A function \(u\in\mathcal A^\nu\) is
\emph{polyhomogeneous of order at least \(\nu\)}, written
\(u\in\mathcal A^\nu_{\mathrm{phg}}\), if there are
exponents \(\nu\le\sigma_0<\sigma_1<\cdots\to\infty\), integers \(N_j\ge0\), and
coefficients \(a_{jp}(\theta,z')\), smooth and of period \(2\pi\) in \(\theta\), such
that
\begin{equation}\label{eq:phg-remainder}
  u-\sum_{\sigma_j<N}\;\sum_{p=0}^{N_j}a_{jp}(\theta,z')\,r^{\sigma_j}(\log r)^p
  \ \in\ \mathcal A^N
  \qquad\text{for every }N .
\end{equation}
One writes \(u\sim\sum_{j,p}a_{jp}(\theta,z')r^{\sigma_j}(\log r)^p\).  Thus
\(\sim\) abbreviates \eqref{eq:phg-remainder}; in particular the expansion may be
differentiated term by term in the edge vector fields, with the differentiated
remainder.
\end{definition}

\begin{definition}[Polyhomogeneous K\"ahler metrics]
\label{def:polyhomogeneous}
Let \(g\) be a K\"ahler metric on \(B\setminus D\) satisfying
\eqref{eq:edge-comparison}, and let \(\omega\) be its K\"ahler form.  We say that
\(g\) is \emph{polyhomogeneous along \(D\)} if the following two conditions hold.
\begin{enumerate}[label=\textup{(\roman*)}]
\item There are a function \(\varphi_0\in C^\infty(D,\R)\), a number
\(\delta>0\), and a K\"ahler potential \(\varphi\), bounded and
plurisubharmonic on \(B\) and smooth on \(B\setminus D\), such that
\(\omega=\ii\,\partial\bar\partial\varphi\) on \(B\setminus D\) and
\[
  \varphi-\varphi_0(z')\in\mathcal A^{1+\delta}_{\mathrm{phg}}
\]
in the sense of Definition~\ref{def:phg}.  Equivalently,
\begin{equation}\label{eq:polyhomogeneous}
  \varphi\ \sim\ \varphi_0(z')
   +\sum_{j\ge1}\;\sum_{p=0}^{N_j}a_{jp}(\theta,z')\,r^{\sigma_j}(\log r)^p,
  \qquad
  1<\sigma_1<\sigma_2<\cdots\to\infty .
\end{equation}
\item The matrix \(G\) of \(g\) in the cone coordinates has a
\(\theta\)-independent leading term: for some matrix-valued smooth function
\(G_0(z')\), necessarily Hermitian and positive definite,
\begin{equation}\label{eq:G-limit}
  G\longrightarrow G_0(z')
  \qquad\text{as }r\to0,\text{ uniformly in }\theta .
\end{equation}
\end{enumerate}
\end{definition}

\begin{remark}[Relation with the K\"ahler--Einstein expansion]\label{rem:jmr}
Definition~\ref{def:polyhomogeneous} deliberately records only the two features of
the polyhomogeneous K\"ahler--Einstein expansion that are used below: the potential
has no nonconstant term of order at most \(1\), and the metric in cone coordinates
has a \(\theta\)-independent limit.  The known expansion contains more information.

In \cite{JMR2016} the unknown potential is measured relative to a reference
K\"ahler cone metric.  By \cite[Theorem~2]{JMR2016} its exponents belong to
\[
  \bigl\{j+k/\beta:j,k\in\N\cup\{0\}\bigr\},
\]
and no logarithms occur at order at most \(2\).  The reference potential has an
expansion of the same type (\cite{JMR2016}, proof of Lemma~4.5), and hence so does
the total potential \(\varphi\).  Moreover, \cite[Proposition~4.4]{JMR2016}
identifies the relevant initial terms.  When \(\tfrac12<\beta<1\),
\[
  \varphi\ \sim\ a_{00}(z')
   +\bigl(a_{01}(z')\sin\theta+b_{01}(z')\cos\theta\bigr)r^{1/\beta}
   +a_{20}(z')\,r^{2}+O(r^{2+\epsilon})
  \qquad\text{for some }\epsilon>0.
\]
When \(\beta<\tfrac12\), the term \(a_{20}(z')r^2\) occurs before the term of
order \(r^{1/\beta}\); when \(\beta=\tfrac12\), the two have the same order.

This more precise list of powers is not imposed in
Definition~\ref{def:polyhomogeneous}.  Indeed, the arguments below use only the
two stated conditions.  To see how the precise expansion implies them, first
observe that no term of order \(r\) can occur.  This is already forced by
\eqref{eq:edge-comparison}: its coefficient \(a(\theta)\) would have to satisfy
\(a''=-\beta^2a\), which has no nonzero \(2\pi\)-periodic solution for
\(0<\beta<1\).  Thus the first nonconstant order is strictly greater than \(1\).
Second, the coefficient of \(r^2\), whose normal complex Hessian produces the
leading normal component of \(G\), depends only on \(z'\).  The only lower-order
term that may depend on \(\theta\), namely the term of order \(r^{1/\beta}\), is
pluriharmonic in the normal variable \(\zeta\) and therefore does not affect that
leading component.  This gives \eqref{eq:G-limit}.
\end{remark}

\begin{lemma}[Scale-invariant estimates]\label{lem:scale-estimates}
Let \(g\) be polyhomogeneous along \(D\).  After shrinking \(B\) there are
\(\Lambda\ge1\) and \(\varepsilon\in(0,1)\) such that at every point of
\(B\setminus D\), with \(r\) as in \eqref{eq:r-rho} and all quantities computed in the
cone coordinates,
\begin{enumerate}[label=\textup{(\roman*)}]
\item\label{it:reg-metric}
  \(\Lambda^{-1}\Id\le G\le\Lambda\Id\)
  \quad\text{and}\quad
  \(|\partial G|\le\Lambda\,r^{\varepsilon-1}\);
\item\label{it:reg-potential}
  \(|\varphi|+|\partial\varphi|\le\Lambda\)
  \quad\text{and}\quad
  \(|\partial^2\varphi|\le\Lambda\,r^{\varepsilon-1}\).
\end{enumerate}
\end{lemma}

\begin{proof}
Write
\begin{equation}\label{eq:d-zeta-edge}
  \partial_\zeta=\frac{e^{-\ii\beta\theta}}{2r}\,L,
  \qquad
  L=r\partial_r-\frac{\ii}{\beta}\,\partial_\theta ,
\end{equation}
so that a \(\zeta\)-derivative is \(r^{-1}\) times an edge operator, whereas
\(\partial_{z^a}\) is itself one.  The two-sided bound on \(G\) is
\eqref{eq:edge-comparison}, and the bound on \(\varphi\) is its boundedness.

Fix \(\sigma'\in(1,\sigma_1)\).  Since \(r^\sigma(\log r)^p=O(r^{\sigma'})\) for
\(\sigma>\sigma'\), \eqref{eq:phg-remainder} makes \(\varphi-\varphi_0(z')\) and each
of its iterated edge derivatives \(O(r^{\sigma'})\); as \(L\varphi_0=0\), the same
bound holds for \(L\varphi\) and for its further edge derivatives.  Next, up to a
phase factor \(e^{-ik\beta\theta}\) which is invisible to the bounds below, \(G\) is
polyhomogeneous, being obtained from \(\varphi\) by two applications of
\eqref{eq:d-zeta-edge} and \(\partial_{z^a}\); by \eqref{eq:edge-comparison} its
expansion carries no term of negative exponent and no logarithm at exponent \(0\),
and by \eqref{eq:G-limit} its coefficient at exponent \(0\) is \(G_0(z')\).  Hence
\(G-G_0(z')\) and its iterated edge derivatives are \(O(r^\mu)\) for some
\(\mu>0\), and \(LG=O(r^\mu)\).

Put \(\varepsilon=\min\{\tfrac12,\ \sigma'-1,\ \mu\}\).  Tangential derivatives are
harmless: \(\partial_{z^a}\varphi\), \(\partial_{z^a}\partial_{z^b}\varphi\) and
\(\partial_{z^a}G\) are \(O(1)\).  By \eqref{eq:d-zeta-edge},
\[
  \partial_\zeta\varphi=O(r^{\sigma'-1}),
  \qquad
  \partial_\zeta\partial\varphi=O(r^{\sigma'-2}),
  \qquad
  \partial_\zeta G=O(r^{\mu-1}),
\]
the first \(O(1)\) and the other two \(O(r^{\varepsilon-1})\).  Since \(r\) is
bounded, the quantities that are \(O(1)\) are also \(O(r^{\varepsilon-1})\).
\end{proof}

From now on, polyhomogeneity is used only through
Lemma~\ref{lem:scale-estimates}; no individual term of the expansion
\eqref{eq:polyhomogeneous} is used again.  In the flat case only
\ref{it:reg-metric} is needed, whereas the jet-bundle argument in
Section~\ref{sec:jet} uses both \ref{it:reg-metric} and
\ref{it:reg-potential}.  By
\eqref{eq:chsc-convention} a chsc metric has constant Ricci curvature, so the
hypotheses of Theorem~\ref{thm:main} are those of Lemma~\ref{lem:scale-estimates}.

\subsection{Christoffel-symbol estimates}

For a K\"ahler metric the Christoffel symbols in holomorphic coordinates are the
entries of \(G^{-1}\partial G\), so Lemma~\ref{lem:scale-estimates}\ref{it:reg-metric}
and \eqref{eq:r-rho} give
\begin{equation}\label{eq:cone-christoffel}
  \widetilde\Gamma^k_{ij}=O(r^{\varepsilon-1})=O(\rho^{-\beta+\beta\varepsilon}).
\end{equation}
Since \(\partial_{z^1}=(z^1)^{\beta-1}\partial_\zeta\), the transformation law for
connections gives
\begin{align}
  \Gamma^1_{11}&=\frac{\beta-1}{z^1}+(z^1)^{\beta-1}\widetilde\Gamma^1_{11},
  &\Gamma^a_{11}&=(z^1)^{2\beta-2}\widetilde\Gamma^a_{11},
  \label{eq:transform-1}\\
  \Gamma^1_{a1}&=\widetilde\Gamma^1_{a1},
  &\Gamma^b_{a1}&=(z^1)^{\beta-1}\widetilde\Gamma^b_{a1},
  \label{eq:transform-2}\\
  \Gamma^1_{ab}&=(z^1)^{1-\beta}\widetilde\Gamma^1_{ab},
  &\Gamma^c_{ab}&=\widetilde\Gamma^c_{ab}.
  \label{eq:transform-3}
\end{align}
The powers of \(z^1\) are written after choosing a branch, but the left-hand sides are
the Christoffel symbols of \(g\) in the coordinates \(z\), and the estimates below are
independent of the branch.  Combining
\eqref{eq:transform-1}--\eqref{eq:transform-3} with \eqref{eq:cone-christoffel}:
\begin{align}
  \Gamma^1_{11}&=\frac{\beta-1}{z^1}+O(\rho^{-1+\beta\varepsilon}),
  &\Gamma^a_{11}&=O(\rho^{\beta-2+\beta\varepsilon}),
  \label{eq:gamma-est-1}\\
  \Gamma^1_{a1}&=O(\rho^{-\beta+\beta\varepsilon}),
  &\Gamma^b_{a1}&=O(\rho^{-1+\beta\varepsilon}),
  \label{eq:gamma-est-2}\\
  \Gamma^1_{ab}&=O(\rho^{1-2\beta+\beta\varepsilon}),
  &\Gamma^c_{ab}&=O(\rho^{-\beta+\beta\varepsilon}).
  \label{eq:gamma-est-3}
\end{align}

\begin{corollary}[Subcritical pole estimates]\label{cor:subcritical}
There are \(\ell_0<1\) and \(\ell_1<2\) such that, locally uniformly in \(z'\),
\begin{equation}\label{eq:subcritical}
  \Gamma^k_{aj}=O(\rho^{-\ell_0}),
  \qquad
  \Gamma^k_{1j}=O(\rho^{-\ell_1}) .
\end{equation}
\end{corollary}

\begin{proof}
For a tangential first lower index the pole exponents occurring in
\eqref{eq:gamma-est-2}--\eqref{eq:gamma-est-3} are at most
\(\beta-\beta\varepsilon\), \(1-\beta\varepsilon\) and
\(2\beta-1-\beta\varepsilon\), each strictly less than \(1\) because \(0<\beta<1\)
and \(\varepsilon>0\).  For first lower index \(1\), the K\"ahler symmetry
\(\Gamma^k_{1a}=\Gamma^k_{a1}\) reduces to the previous case except for
\(\Gamma^1_{11}=O(\rho^{-1})\) and
\(\Gamma^a_{11}=O(\rho^{-(2-\beta-\beta\varepsilon)})\), with
\(2-\beta-\beta\varepsilon<2\).  Take \(\ell_0,\ell_1\) to be the corresponding
maxima.
\end{proof}

\begin{remark}\label{rem:gain}
The positive exponent in Lemma~\ref{lem:scale-estimates}\ref{it:reg-metric} is
essential, and it is exactly what the \(\theta\)-independence of the limit
\eqref{eq:G-limit} buys.
The bound \(G=O(1)\) alone gives only the borderline estimate
\(|\partial G|=O(r^{-1})\); the gain \(r^\varepsilon\) is what moves the transformed
growth orders in Corollary~\ref{cor:subcritical} strictly below the integers \(1\)
and \(2\).
\end{remark}

\begin{remark}[Related Christoffel-symbol estimates]\label{rem:related-christoffel}
Estimates for the orders of growth of Christoffel symbols near a cone divisor
appear in Keller--Zheng
\cite[Proposition~2.3, eq.~(2.9)]{KellerZheng2018}.  Under their
\(C^{3,\alpha,\beta}\) assumptions, they describe the behaviour of the individual
Christoffel symbols after multiplication by the natural powers of \(|z^1|\); in
particular, they separate the model singular term
\[
  \Gamma^1_{11}=\frac{\beta-1}{z^1}+\text{remainder}.
\]
Related estimates are used by Song--Wang
\cite[\S4.1]{SongWang2016}.  Starting from the polyhomogeneous expansion of
the metric, they estimate its coefficients, its inverse, and the associated Chern
connection forms for a Chern--Weil calculation.

The purpose of \eqref{eq:gamma-est-1}--\eqref{eq:gamma-est-3} is different.  We
use them only to obtain the strict inequalities in
Corollary~\ref{cor:subcritical}; in the flat case, holomorphicity then turns these
growth estimates into the logarithmic extension of
Section~\ref{sec:log-extension}.
\end{remark}

\section{The Levi--Civita connection is logarithmic}\label{sec:log-extension}

Let \(E\) be a holomorphic vector bundle on a complex manifold and \(D=\{z^1=0\}\) a
smooth divisor.  A connection \(\nabla\) on \(E\) is \emph{logarithmic along \(D\)}
if, in a holomorphic frame,
\begin{equation}\label{eq:log-connection}
  \nabla=d+A,
  \qquad
  A=A_1\frac{dz^1}{z^1}+\sum_{a=2}^nA_a\,dz^a
\end{equation}
with every \(A_i\) holomorphic across \(D\); its \emph{residue} is
\(\Res_D(\nabla)=A_1|_D\in H^0(D,\End(E|_D))\).  Neither notion depends on the local
equation for \(D\) or on the holomorphic frame, the residue being conjugated under a
change of frame.  See \cite[Chapter~II]{Deligne1970} or
\cite[Sections~3--4]{NovikovYakovenko2004}.

For a K\"ahler metric the Levi--Civita connection preserves \(T^{1,0}\) and agrees
there with the Chern connection of \((T^{1,0},g)\); in the coordinate frame
\begin{equation}\label{eq:christoffel-formula}
  \Gamma^k_{ij}=g^{k\bar\ell}\partial_ig_{j\bar\ell},
  \qquad
  \nabla_{\partial_i}\partial_j=\Gamma^k_{ij}\partial_k,
\end{equation}
the K\"ahler condition being the torsion-freeness \(\Gamma^k_{ij}=\Gamma^k_{ji}\).
The curvature of the Chern connection is \(-\partial_{\bar\ell}\Gamma^k_{ij}\), so
\(g\) is flat if and only if every \(\Gamma^k_{ij}\) is holomorphic.

The following result is elementary and left to the reader.

\begin{lemma}\label{lem:extension}
Let \(U\subset\C^n\) be a polydisc, \(D=\{z^1=0\}\), and \(f\in\cO(U\setminus D)\) with
\(|f(z)|\le C|z^1|^{-\ell}\) for some \(C>0\) and \(\ell>0\).  If
\(m\in\N\cup\{0\}\) and \(\ell<m+1\), then \((z^1)^mf\) extends holomorphically
across \(D\).
\end{lemma}

The main result of this section is the following.

\begin{theorem}[Logarithmic extension]\label{thm:log-extension}
Let \(g\) be a flat K\"ahler metric on \(B\setminus D\), satisfying
\eqref{eq:edge-comparison} with \(0<\beta<1\), and polyhomogeneous along \(D\) in the
sense of Definition~\ref{def:polyhomogeneous}.  Let \(\nabla\) be its Levi--Civita
connection on \(T^{1,0}(B\setminus D)\).  After shrinking \(B\), \(\nabla\) extends
as a logarithmic connection on \(T^{1,0}B\) whose residue, in the
coordinate frame \((\partial_1,\dots,\partial_n)\), has the form
\begin{equation}\label{eq:residue-block}
  \Res_D(\nabla)=\begin{pmatrix}\beta-1&0\\ v&0\end{pmatrix}
\end{equation}
for a holomorphic column \(v\) along \(D\).
\end{theorem}

\begin{proof}
Every \(\Gamma^k_{ij}\) is holomorphic on \(B\setminus D\).  Applying
Lemma~\ref{lem:extension} with \(m=0\) to the first family in
\eqref{eq:subcritical} shows that each \(\Gamma^k_{aj}\) extends holomorphically
across \(D\), and with \(m=1\) to the second family that each \(z^1\Gamma^k_{1j}\)
does.  Hence \(\nabla\) has the shape \eqref{eq:log-connection}.

By definition \(\bigl(\Res_D(\nabla)\bigr)^k{}_j=z^1\Gamma^k_{1j}|_D\).  For
\(j=a\ge2\), torsion-freeness gives \(\Gamma^k_{1a}=\Gamma^k_{a1}\), which is
holomorphic across \(D\), so \(z^1\Gamma^k_{1a}|_D=0\): all columns of the residue
except the first vanish.  By \eqref{eq:gamma-est-1},
\(z^1\Gamma^1_{11}=\beta-1+O(\rho^{\beta\varepsilon})\), so the extension of
\(z^1\Gamma^1_{11}\) equals \(\beta-1\) on \(D\).
For \(a\ge2\), the holomorphic extension of \(z^1\Gamma^a_{11}\) may
have a nonzero restriction to \(D\); setting
\[
v^a=\left.z^1\Gamma^a_{11}\right|_D
\]
defines the holomorphic column \(v=(v^2,\dots,v^n)^{\mathsf T}\).
This is
\eqref{eq:residue-block}.
\end{proof}

\begin{remark}\label{rem:beta-two}
For \(\beta>1\) the Levi--Civita connection need not extend logarithmically.  Let
\[
  \pi(u,v)=\bigl(u^2+v^2,\,v\bigr)
\]
be the degree two cover of \(\C^2\) branched along the parabola
\(\{z^1=(z^2)^2\}\) and ramified along \(D=\{u=0\}\).  By construction the pullback
\(g=\pi^*g_{\mathrm{Euc}}\) is flat on the complement of \(D\) and has cone angle
\(4\pi\) along \(D\); in particular it satisfies \eqref{eq:edge-comparison} with
\(\beta=2\).  Its Levi--Civita connection is the pullback of the trivial one, so
differentiating the coefficients of
\[
  \pi_*\partial_u=2u\,\partial_{z^1},
  \qquad
  \pi_*\partial_v=2v\,\partial_{z^1}+\partial_{z^2}
\]
and inverting \(\pi_*\) gives
\[
  \nabla_{\partial_u}\partial_u=\nabla_{\partial_v}\partial_v=\frac1u\,\partial_u,
  \qquad
  \nabla_{\partial_u}\partial_v=0 .
\]
In the frame \((\partial_u,\partial_v)\) the connection form is therefore
\[
  A=\frac1u\begin{pmatrix}du&dv\\[2pt]0&0\end{pmatrix} .
\]
The coefficient of \(dv\) has a pole along \(D\), so \(A\) is not of the form
\eqref{eq:log-connection}.
\end{remark}

\section{Normal form for torsion-free flat logarithmic connections}
\label{sec:normal-form}

We use throughout the convention \(\nabla=d+A\), so that the residue is the residue
of \(A\).  In this section the coordinate frame is indexed by \(1,\dots,n\), and
\(E_{11}\) denotes the \(n\times n\) matrix with a single \(1\) in the upper-left
entry.

\begin{theorem}[Coordinate normal form]\label{thm:connection-normal-form}
Let \(U\subset\C^n\) be a ball centred at the origin, \(D=\{z^1=0\}\), and let
\(\nabla\) be a flat torsion-free logarithmic connection on \(TU\).  Put
\(R=\Res_D(\nabla)\).  Put \(\lambda=\tr R\) and suppose that
\(\lambda\in\C\setminus\Z\).  Then, after shrinking \(U\), there are holomorphic
coordinates \(w=(w^1,\dots,w^n)\) centred at the origin with \(D=\{w^1=0\}\) and
\begin{equation}\label{eq:connection-normal-form}
  \nabla=d+\lambda E_{11}\frac{dw^1}{w^1}
\end{equation}
in the coordinate frame; equivalently
\(\nabla_{\partial_{w^1}}\partial_{w^1}=\lambda(w^1)^{-1}\partial_{w^1}\) and all other
covariant derivatives of coordinate fields vanish.
\end{theorem}

\begin{remark}
For a flat logarithmic connection, the conjugacy class of the residue is locally
constant along the smooth part of the polar divisor; see
\cite[Theorem~4.2]{NovikovYakovenko2004}.  Since \(D\) is connected, its
eigenvalues and trace are constant on \(D\).  Thus the constancy of
\(\lambda=\tr R\) in Theorem~\ref{thm:connection-normal-form} is automatic and is
not an additional hypothesis.
\end{remark}

\begin{remark}[Torsion and the residue]\label{rem:torsion-residue}
Torsion-freeness implies
\[
  TD\subseteq\ker R;
\]
see \cite[Lemma~A.10]{deBorbonPanov2021}.  If \(\lambda=\tr R\neq0\), it follows
that
\[
  \ker R=TD,\qquad \operatorname{rank}R=1,\qquad R^2=\lambda R.
\]
In particular, \(R\) is diagonalizable with eigenvalues \(\lambda,0,\dots,0\), and
its \(\lambda\)-eigenline is transverse to \(TD\).
\end{remark}

\subsection{Gauge normalisation and parallel foliations}

We use the following normal form theorem for flat logarithmic connections, for a proof see  \cite[Theorem~A.11]{deBorbonPanov2021}.

\begin{theorem}[Nonresonant gauge normal form]\label{thm:gauge}
Let \(\nabla=d+A\) be a flat logarithmic connection on a holomorphic vector bundle
over a polydisc \(U\), with polar divisor \(D=\{z^1=0\}\).  Assume that
\(R=\Res_D(\nabla)\) is diagonalizable and that no two of its eigenvalues differ by
a nonzero integer.  Then, after shrinking \(U\), there is a holomorphic frame in
which \(\nabla=d+R_0\,dz^1/z^1\) with \(R_0\) constant and diagonal.
\end{theorem}

The next result is a particular case of
\cite[Lemmas~4.35 and~4.47]{deBorbonPanov2021}.  We include the proof for
completeness.

\begin{lemma}[Parallel foliations and product coordinates]\label{lem:product-coordinates}
Assume the hypotheses of Theorem~\ref{thm:connection-normal-form}.  After shrinking
\(U\) there are a holomorphic frame
\((s_1,\dots,s_n)\) of \(TU\) with
\begin{equation}\label{eq:gauge-frame}
  \nabla s_1=\lambda\frac{dz^1}{z^1}\otimes s_1,
  \qquad
  \nabla s_a=0\quad(a\ge2),
\end{equation}
and holomorphic coordinates \((u,y^2,\dots,y^n)\) centred at the origin such that,
setting \(\mathcal N=\cO s_1\) and \(\mathcal T=\Span_{\cO}\{s_2,\dots,s_n\}\),
\[
  D=\{u=0\},
  \qquad
  \mathcal N=\Span\{\partial_u\},
  \qquad
  \mathcal T=\Span\{\partial_{y^2},\dots,\partial_{y^n}\}.
\]
\end{lemma}

\begin{proof}
Since \(\lambda\notin\Z\) we have \(\lambda\neq0\), so by
Remark~\ref{rem:torsion-residue} the residue is diagonalizable with eigenvalues
\(\lambda,0,\dots,0\); the nonzero differences of eigenvalues are \(\pm\lambda\notin\Z\),
so Theorem~\ref{thm:gauge} applies and yields a holomorphic frame satisfying
\eqref{eq:gauge-frame}.  In this frame \(A=\lambda E_{11}\,dz^1/z^1\), so
\(R=\lambda E_{11}\); combining with Remark~\ref{rem:torsion-residue},
\begin{equation}\label{eq:distributions-on-D}
  \mathcal T|_D=\ker R=TD,
  \qquad
  \mathcal N|_D=\operatorname{im}R .
\end{equation}

The line distribution \(\mathcal N\) is integrable.  For holomorphic sections
\(X,Y\) of \(\mathcal T\), parallelism and torsion-freeness give
\([X,Y]=\nabla_XY-\nabla_YX\in\mathcal T\) on \(U\setminus D\); both sides extend
holomorphically across \(D\), so \(\mathcal T\) is integrable as well.  By
\eqref{eq:distributions-on-D} the leaf of \(\mathcal T\) through the origin is \(D\).
The holomorphic Frobenius theorem provides a first integral \(u\) of \(\mathcal T\)
with zero set \(D\), and \(n-1\) independent first integrals \(y^2,\dots,y^n\) of
\(\mathcal N\).  The distributions being complementary, \(du,dy^2,\dots,dy^n\) are
independent at the origin, and in the resulting coordinates \(\mathcal N=\Span\{\partial_u\}\)
and \(\mathcal T=\Span\{\partial_{y^a}\}\).
\end{proof}

The next result is a particular case of
\cite[Lemmas~4.55, 4.57 and~4.58]{deBorbonPanov2021}.  We include the proof for
completeness.

\begin{lemma}\label{lem:mixed-vanishing}
In the situation of Lemma~\ref{lem:product-coordinates},
\(\nabla_{\partial_u}\partial_{y^a}=\nabla_{\partial_{y^a}}\partial_u=0\) for all
\(a\ge2\).  Moreover, after a change of the tangential variables alone, one may
assume \(s_a=\partial_{y^a}\) and hence \(\nabla\partial_{y^a}=0\) for \(a\ge2\).
\end{lemma}

\begin{proof}
Parallelism of the two distributions gives
\(\nabla_{\partial_u}\partial_{y^a}\in\mathcal T\) and
\(\nabla_{\partial_{y^a}}\partial_u\in\mathcal N\); the coordinate fields commute, so
torsion-freeness makes these two vectors equal, and their common value lies in
\(\mathcal N\cap\mathcal T=0\).

Write \(s_a=\sum_bc^b_a(u,y)\partial_{y^b}\).  By \eqref{eq:gauge-frame} and the
first assertion, \(0=\nabla_{\partial_u}s_a=\sum_b(\partial_uc^b_a)\partial_{y^b}\),
so the \(c^b_a\) depend only on \(y\).  The fields \(s_a\) commute, since
\([s_a,s_b]=\nabla_{s_a}s_b-\nabla_{s_b}s_a=0\) on \(U\setminus D\) and the identity
extends across \(D\); a commuting holomorphic frame is a coordinate frame, and since
its coefficients depend only on \(y\) the rectification changes only the tangential
variables.  In the new coordinates \(s_a=\partial_{y^a}\), and \eqref{eq:gauge-frame}
gives \(\nabla\partial_{y^a}=0\).
\end{proof}

\begin{proposition}[Local product decomposition]\label{prop:connection-product}
In the coordinates of Lemma~\ref{lem:mixed-vanishing} there is a holomorphic
function \(q\) of one variable with
\begin{equation}\label{eq:one-dimensional-factor}
  \nabla_{\partial_u}\partial_u=\Bigl(\frac{\lambda}{u}+q(u)\Bigr)\partial_u,
\end{equation}
all other covariant derivatives of the coordinate fields being zero.
\end{proposition}

\begin{proof}
Since \(\mathcal N\) is parallel and generated by \(\partial_u\), we have
\(\nabla_{\partial_u}\partial_u=\gamma\,\partial_u\) for a meromorphic \(\gamma\).  By
Lemma~\ref{lem:mixed-vanishing}, \(\nabla_{\partial_{y^a}}\partial_u=0\), so flatness
gives
\(0=\nabla_{\partial_{y^a}}\nabla_{\partial_u}\partial_u
   -\nabla_{\partial_u}\nabla_{\partial_{y^a}}\partial_u
  =(\partial_{y^a}\gamma)\partial_u\),
and \(\gamma\) depends only on \(u\).  In the frame
\((\partial_u,\partial_{y^2},\dots,\partial_{y^n})\) all entries of the connection
matrix vanish except the one in the \(\partial_u\) slot, which is \(\gamma\,du\);
hence \(\nabla\) is logarithmic along \(\{u=0\}\) with residue
\((u\gamma)|_{u=0}E_{11}\).  Its trace is \(\lambda\) by
the invariance of the residue trace under holomorphic changes of coordinates, so
\(\gamma=\lambda/u+q\) with \(q\) holomorphic.
\end{proof}

\subsection{The one-dimensional case}

The next result is the affine form of the classical normal coordinate for a flat cone; compare
\cite[Proposition~2]{Troyanov1986} and, for the resonant alternatives,
\cite[Proposition~3.1]{DeroinGuillot2023}.  

\begin{proposition}\label{prop:one-dimensional-normal-form}
Let \(\nabla\) be a logarithmic connection on the tangent bundle of a disc, written
in a coordinate \(u\) as
\(\nabla_{\partial_u}\partial_u=\bigl(\lambda u^{-1}+q(u)\bigr)\partial_u\) with \(q\)
holomorphic.  If \(\lambda\notin\{-1,-2,-3,\dots\}\), there is a holomorphic
coordinate \(w\) with \(w(0)=0\) and \(w'(0)\neq0\) such that
\(\nabla_{\partial_w}\partial_w=\lambda w^{-1}\partial_w\).
\end{proposition}

\begin{proof}
Put \(\mu=1+\lambda\), so that
\(\mu\notin\{0,-1,-2,\dots\}\).  Let \(Q\) be a holomorphic primitive of
\(q\), write \(e^{Q(u)}=\sum_{k\ge0}a_ku^k\) with \(a_0\neq0\), and set
\[
  b_k=\frac{\mu}{\mu+k}a_k,\qquad f(u)=\sum_{k\ge0}b_ku^k .
\]
The hypothesis on \(\mu\) makes every denominator nonzero and the multipliers
bounded, so \(f\) converges on the same disc and \(f(0)=a_0\neq0\); comparing
coefficients gives
\begin{equation}\label{eq:f-ode}
  f(u)+\frac{u}{\mu}f'(u)=e^{Q(u)} .
\end{equation}
After shrinking the disc, choose a branch of \(f^{1/\mu}\) and set
\(w=uf(u)^{1/\mu}\), so that \(w'(0)=f(0)^{1/\mu}\neq0\).  Differentiating
\(w^\mu=u^\mu f(u)\) on the universal cover of the punctured disc and using
\eqref{eq:f-ode} yields
\begin{equation}\label{eq:differential-identity}
  w^{\mu-1}\,dw=u^{\mu-1}e^{Q(u)}\,du .
\end{equation}
Writing \(p(u)=\lambda u^{-1}+q(u)\), the transformation law for
\(\partial_w=(w')^{-1}\partial_u\) gives
\(\nabla_{\partial_w}\partial_w=\bigl(p/w'-w''/(w')^2\bigr)\partial_w\), while the
logarithmic derivative of \eqref{eq:differential-identity} gives
\((\mu-1)w'/w+w''/w'=p\).  Substituting proves the claim.
\end{proof}

\begin{proof}[Proof of Theorem~\ref{thm:connection-normal-form}]
By Lemma~\ref{lem:product-coordinates},
Lemma~\ref{lem:mixed-vanishing} and Proposition~\ref{prop:connection-product} there
are coordinates \((u,y^2,\dots,y^n)\) with \(D=\{u=0\}\) in which
\eqref{eq:one-dimensional-factor} holds and all other covariant derivatives of
coordinate fields vanish.  Since \(\lambda\notin\Z\),
Proposition~\ref{prop:one-dimensional-normal-form} applies and replaces \(u\) by a
coordinate \(w^1\) with
\(\nabla_{\partial_{w^1}}\partial_{w^1}=\lambda(w^1)^{-1}\partial_{w^1}\).
Keeping \(w^a=y^a\) and noting that \(w^1\) depends only on \(u\), the mixed and
tangential covariant derivatives remain zero.
\end{proof}

\section{Proof of the classification in the flat case}\label{sec:classification}

\begin{proof}[Proof of Theorem~\ref{thm:main}(i)]
Let \(\nabla\) be the Levi--Civita connection of \(g\).  By
Theorem~\ref{thm:log-extension} it extends to a flat torsion-free logarithmic
connection on \(T^{1,0}B\) with \(\lambda=\tr\Res_D(\nabla)=\beta-1\in(-1,0)\), so
\(\lambda\notin\Z\) and Theorem~\ref{thm:connection-normal-form} provides
holomorphic coordinates \(w\), centred at the origin with \(D=\{w^1=0\}\), in which
\begin{equation}\label{eq:normal-christoffel}
  \Gamma^1_{11}=\frac{\beta-1}{w^1},
  \qquad
  \Gamma^k_{ij}=0\quad\text{for all other triples }(i,j,k).
\end{equation}

Write \(g=g_{i\bar j}\,dw^id\bar w^j\).  The coordinate fields
\(\partial_{w^a}\), \(a\ge2\), are \(\nabla\)-parallel, so \(\nabla g=0\) gives
\(g_{a\bar b}=P_{a\bar b}\) for a constant positive definite Hermitian matrix \(P\).

At any point of \(B\setminus D\), the holonomy of \(\nabla\) around a positive loop
about \(D\) fixes the fields \(\partial_{w^a}\) and multiplies \(\partial_{w^1}\) by
\(e^{-2\pi\ii(\beta-1)}\neq1\).  Since the holonomy preserves \(g\), the normal and
tangential directions are orthogonal.  Hence \(g_{1\bar a}=g_{a\bar1}=0\).  Finally,
metric compatibility and \eqref{eq:normal-christoffel} give
\[
  \partial_{w^a}g_{1\bar1}=\partial_{\bar w^a}g_{1\bar1}=0\quad(a\ge2),
  \qquad
  \partial_{w^1}g_{1\bar1}=\frac{\beta-1}{w^1}g_{1\bar1},
  \qquad
  \partial_{\bar w^1}g_{1\bar1}=\frac{\beta-1}{\bar w^1}g_{1\bar1}.
\]
Thus \(g_{1\bar1}\) depends only on \(w^1\) and \(\bar w^1\), and the last two
identities give
\[
  d\bigl(|w^1|^{2-2\beta}g_{1\bar1}\bigr)=0.
\]
Equivalently, on any sector where a branch is chosen,
\((w^1)^{1-\beta}\partial_{w^1}\) is parallel and the expression in parentheses is
its squared norm.  Hence \(g_{1\bar1}=t\,|w^1|^{2\beta-2}\) for a constant \(t>0\),
and therefore
\[
  g=t\,|w^1|^{2\beta-2}|dw^1|^2+P_{a\bar b}\,dw^ad\bar w^b .
\]
A constant complex linear change of the tangential coordinates turns \(P\) into the
identity, and replacing \(w^1\) by \(t^{1/(2\beta)}w^1\) turns the first summand into
\(|w^1|^{2\beta-2}|dw^1|^2\).  This is \eqref{eq:main-flat}.
\end{proof}

\section{Proof of the classification in nonzero curvature}\label{sec:jet}

Let \(g\) be as in Theorem~\ref{thm:main} with \(c\neq0\), and set
\(\sigma=\operatorname{sgn}c\).  After multiplying \(g\) by a positive constant,
we may assume that its holomorphic sectional curvature is \(2\sigma\).  Let
\(\omega\) be its K\"ahler form and choose the polyhomogeneous K\"ahler potential
\(\varphi\) in Definition~\ref{def:polyhomogeneous}.  Put
\(h_\sigma=e^{-\sigma\varphi}\), regarded as a Hermitian metric on the trivial
holomorphic line bundle \(\mathcal O_B\), and let \(\nabla^{h_\sigma}\) be its Chern
connection.

We briefly recall the bundle of 1-jets in this setting.  Two germs of holomorphic
functions at \(p\) define the same element of \(J^1(\mathcal O_B)_p\) if they have
the same value and first derivative at \(p\).  We may therefore identify
\(J^1(\mathcal O_B)_p\) with \(\C\oplus T_p^{*1,0}B\), under which
\(j_p(f)=(f(p),df_p)\); in particular, \(J^1(\mathcal O_B)\) has rank \(n+1\).  It
fits into the exact sequence
\(0\to T^{*1,0}B\xrightarrow{\iota}J^1(\mathcal O_B)
\xrightarrow{\pi}\mathcal O_B\to0\), where \(\iota(\alpha)=(0,\alpha)\) and
\(\pi(a,\alpha)=a\).  Equivalently, \(\iota_p(\alpha)=j_p(f)\) for any holomorphic
function \(f\) such that \(f(p)=0\) and \(df_p=\alpha\).

For a holomorphic function \(f\), let \(j(f)\) be the holomorphic section
\(p\mapsto j_p(f)\) of \(J^1(\mathcal O_B)\); thus \(j(f)=(f,df)\).  The map \(j\)
is a first-order differential operator, not a vector-bundle map
\(\mathcal O_B\to J^1(\mathcal O_B)\): we have \(\pi(j(f))=f\), but \(j\) is not
\(\mathcal O_B\)-linear.

Set \(E=J^1(\mathcal O_B)\) and define the Hermitian form \(\mathcal H\) on \(E\) by
\begin{equation}\label{eq:jet-hermitian-form}
  |j(f)|^2_{\mathcal H}
  =|f|^2_{h_\sigma}
   +\sigma|\nabla^{h_\sigma}f|^2_{g\otimes h_\sigma}.
\end{equation}
For \(\sigma=1\), this is the Hermitian metric of
\cite[Definition~4.1]{deBorbon2025}; for \(\sigma=-1\), it is the Hermitian form of
\cite[Definition~4.4]{deBorbon2025}. The form \(\mathcal H\) is positive definite when
\(\sigma=1\) and has signature \((1,n)\) when \(\sigma=-1\).  By
\cite[Proposition~6.1]{deBorbon2025}, its Chern connection \(\mathcal D\) is flat on
\(B\setminus D\).  The holomorphic bundle \(E\) is defined on all of \(B\), whereas
\(h_\sigma\) is smooth only on \(B\setminus D\) and \(\mathcal H\) and \(\mathcal D\)
are initially defined there.  The content of
Theorem~\ref{thm:jet-log-normal-form} is that \(\mathcal D\) extends across \(D\) as
a logarithmic connection.

We use the holomorphic frame
\(e=(e_0,e_1,\dots,e_n)=(j(1),j(z^1),\dots,j(z^n))\) of \(E\).  Indeed,
\(j(1)=(1,0)\) and \(j(z^i)=(z^i,dz^i)\), so these sections form a frame.  In this
frame,
\begin{equation}\label{eq:first-jet-in-frame}
  j(f)=\left(f-\sum_{i=1}^nz^i\partial_if\right)j(1)
       +\sum_{i=1}^n(\partial_if)j(z^i).
\end{equation}
We index the entries of endomorphisms of \(E\) by \(0,1,\dots,n\).  Thus, in this
section, \(E_{11}=\operatorname{diag}(0,1,0,\dots,0)\), with its \(1\) in the
second diagonal position, corresponding to the \(j(z^1)\) slot.
This differs from the \(n\times n\) matrix unit used in
Section~\ref{sec:normal-form}.  More generally, \(E_{ij}\) denotes the standard
matrix unit, with indices \(0,\dots,n\).  Connection matrices are defined by
\(\mathcal De=e\,A\), so that \(\mathcal D=d+A\).

\subsection{Estimate for the Chern connection in cone coordinates}

On a simply connected sector where \(\zeta\) is defined, consider the holomorphic
frame \(\widetilde e=\bigl(j(1),j(\zeta),j(z^2),\dots,j(z^n)\bigr)\), obtained from
\(e\) by replacing \(j(z^1)\) with \(j(\zeta)\).  Let \(S\) be the corresponding
change-of-frame matrix, so that \(\widetilde e=eS\).

\begin{lemma}\label{lem:jet-cone-estimate}
The connection matrix \(\widetilde A\) of \(\mathcal D\) in the frame
\(\widetilde e\) satisfies
\begin{equation}\label{eq:jet-cone-estimate}
  \widetilde A=O(\rho^{-\beta+\varepsilon'}),
  \qquad
  \varepsilon'=\beta\varepsilon>0 .
\end{equation}
\end{lemma}

\begin{proof}
Compatibility of the Chern connection with \(\mathcal H\), together with
\(\bar\partial\widetilde e=0\), gives
\(\partial\mathcal H=\widetilde B\,\mathcal H\) with
\(\widetilde B=\widetilde A^{\mathsf T}\), whence
\begin{equation}\label{eq:chern-from-H}
  |\widetilde A|\le\|\mathcal H^{-1}\|\;|\partial\mathcal H| ,
\end{equation}
all norms being taken in the cone coordinates.  It suffices therefore to bound
\(\mathcal H^{-1}\) and \(\partial\mathcal H\).

In the cone coordinates the entries of \(\mathcal H\) in the frame \(\widetilde e\)
are universal algebraic expressions in the coordinates \((\zeta,z')\), in
\(h_\sigma=e^{-\sigma\varphi}\), in the components of \(g\) and \(g^{-1}\), and in
those of \(\partial\varphi\) and \(\overline{\partial\varphi}\); indeed
\(\nabla^{h_\sigma}f=\partial f-\sigma f\,\partial\varphi\) for a holomorphic \(f\),
and \(\widetilde e\) consists of the jets of \(1,\zeta,z^a\).  All of these are
bounded by Lemma~\ref{lem:scale-estimates}, and each has \(\partial\)-derivative
\(O(r^{\varepsilon-1})\): the coordinates because their derivatives are constant,
\(h_\sigma\) because \(|\partial\varphi|\le\Lambda\), the components of
\(g^{\pm1}\) by
Lemma~\ref{lem:scale-estimates}\ref{it:reg-metric}, those of
\(\partial\varphi\) by Lemma~\ref{lem:scale-estimates}\ref{it:reg-potential}, and
those of its conjugate by Lemma~\ref{lem:scale-estimates}\ref{it:reg-metric} again,
since
\(\partial_i\overline{\partial_j\varphi}=g_{i\bar j}\).  Hence
\(|\partial\mathcal H|=O(r^{\varepsilon-1})\).

For the uniform nondegeneracy, use the smooth splitting of the first jet exact
sequence induced by the Chern connection \(\nabla^{h_\sigma}\), in which
\(j(f)\mapsto(f,\nabla^{h_\sigma}f)\).  In this splitting \(\mathcal H\) is block
diagonal, with blocks \(h_\sigma\) and
\(\sigma(g^{-1}\otimes h_\sigma)\).  The matrix of \(\widetilde e\) in the splitting
has columns \((1,-\sigma\partial\varphi)\),
\((\zeta,d\zeta-\sigma\zeta\,\partial\varphi)\) and
\((z^a,dz^a-\sigma z^a\,\partial\varphi)\); subtracting \(\zeta\), respectively
\(z^a\), times the first column from the others turns it into
\(\left(\begin{smallmatrix}1&0\\-\sigma\partial\varphi&\Id\end{smallmatrix}\right)\).
It is therefore a product of two unipotent matrices whose entries are \(\zeta\),
\(z^a\) and \(\partial\varphi\), all bounded; since \(h_\sigma\) and
\(g^{\pm1}\) are bounded, so are \(\mathcal H\) and \(\mathcal H^{-1}\).  Now
\eqref{eq:chern-from-H} gives \(\widetilde A=O(r^{\varepsilon-1})\), and
\(r=\rho^\beta/\beta\) turns this into \eqref{eq:jet-cone-estimate}.
\end{proof}

\subsection{Logarithmic extension and gauge normal form}

\begin{theorem}\label{thm:jet-log-normal-form}
The connection \(\mathcal D\) extends across \(D\) as a flat logarithmic connection on
\(E\).  Its residue \(\mathcal R=\Res_D(\mathcal D)\) has rank one and satisfies
\begin{equation}\label{eq:jet-residue-polynomial}
  \mathcal R^2=(1-\beta)\mathcal R,
  \qquad
  \operatorname{Spec}(\mathcal R)=\{1-\beta,0,\dots,0\},
\end{equation}
and after shrinking \(B\) there is a holomorphic frame of \(E\) in which
\begin{equation}\label{eq:jet-log-normal-form}
  \mathcal D=d+(1-\beta)E_{11}\frac{dz^1}{z^1}.
\end{equation}
\end{theorem}

\begin{proof}
Recall that \(S\) is the change-of-frame matrix defined by \(\widetilde e=eS\).
Applying \eqref{eq:first-jet-in-frame} to \(f=\zeta\) gives
\begin{equation}\label{eq:jet-frame-change}
  j(\zeta)=(1-\beta)\zeta\,j(1)+(z^1)^{\beta-1}j(z^1).
\end{equation}
Thus \(S\) is the identity except in its column \(1\), where
\(S_{01}=(1-\beta)\zeta\) and \(S_{11}=(z^1)^{\beta-1}\); correspondingly \(S^{-1}\)
is the identity except for \((S^{-1})_{01}=-\bigl((1-\beta)/\beta\bigr)z^1\) and
\((S^{-1})_{11}=(z^1)^{1-\beta}\).  In particular row \(1\) of \(S\) and column \(1\)
of \(S^{-1}\) carry the powers \(\rho^{\beta-1}\) and \(\rho^{1-\beta}\), and all
other entries are \(O(1)\) or \(O(\rho^\beta)\).  Since \(\widetilde e=eS\),
\begin{equation}\label{eq:jet-gauge-change}
  A=S\widetilde AS^{-1}-dS\,S^{-1}.
\end{equation}

A direct computation from \eqref{eq:jet-frame-change} gives
\begin{equation}\label{eq:jet-gauge-part}
  -dS\,S^{-1}=(1-\beta)E_{11}\frac{dz^1}{z^1}-(1-\beta)E_{01}\,dz^1 ,
\end{equation}
whose second term is holomorphic.  For the first term of
\eqref{eq:jet-gauge-change} we combine \eqref{eq:jet-cone-estimate} with the powers
just listed, remembering that \(d\zeta=(z^1)^{\beta-1}dz^1\).  Writing
\(S\widetilde AS^{-1}=\Psi_1\,dz^1+\sum_a\Psi_a\,dz^a\), the resulting pole exponents
are as follows: every entry of \(\Psi_a\) is \(O(\rho^{-\ell_T})\) with
\[
  \ell_T=\max\{\beta,\,1,\,2\beta-1\}-\varepsilon'<1 ,
\]
while for \(\Psi_1\)
\begin{equation}\label{eq:jet-entrywise}
  (\Psi_1)_{1j}=
  \begin{cases}
    O(\rho^{-(2-\beta-\varepsilon')}) & j\neq1,\\
    O(\rho^{-(1-\varepsilon')}) & j=1,
  \end{cases}
  \qquad
  (\Psi_1)_{ij}=O(\rho^{-(1-\varepsilon')})\ \ (i\neq1),
\end{equation}
all exponents being \(<2\), and \(<1\) except for the entries \((1,j)\) with
\(j\neq1\).

The matrix \(A\) is holomorphic on \(B\setminus D\), because \(\mathcal D\) is a
Chern connection in a holomorphic frame and its curvature \(\bar\partial A\)
vanishes.  Lemma~\ref{lem:extension} with \(m=0\) applied to the entries of
\(\Psi_a\), and with \(m=1\) applied to those of \(\Psi_1\), therefore shows that the
\(dz^a\)-coefficients of \(A\) extend holomorphically across \(D\) and that its
\(dz^1\)-coefficient has at most a simple pole.  Thus \(\mathcal D\) is logarithmic,
and its flatness extends across \(D\) because the meromorphic curvature vanishes on a
dense open set.

By \eqref{eq:jet-entrywise} the entries of \(z^1\Psi_1\) tend to zero on \(D\) except
possibly those in row \(1\).  With \eqref{eq:jet-gauge-part} this gives
\begin{equation}\label{eq:jet-residue-block}
  \mathcal R=(1-\beta)E_{11}+\sum_{j\neq1}v_jE_{1j}
\end{equation}
for holomorphic functions \(v_j\) on \(D\): all rows of \(\mathcal R\) vanish except
row \(1\), whose diagonal entry is \(1-\beta\).  Hence \(\mathcal R\) has rank one and
\(\mathcal R^2=(1-\beta)\mathcal R\), which is \eqref{eq:jet-residue-polynomial};
since \(1-\beta\neq0\), \(\mathcal R\) is diagonalizable.  The only nonzero
difference of eigenvalues of \(\mathcal R\) is \(\pm(1-\beta)\in(-1,1)\setminus\{0\}\),
so no two eigenvalues differ by a nonzero integer and Theorem~\ref{thm:gauge}
applies, giving a holomorphic frame in which \(\mathcal D=d+R_0\,dz^1/z^1\) with
\(R_0\) constant, diagonal and conjugate to \(\mathcal R\); this is
\eqref{eq:jet-log-normal-form}.
\end{proof}

\subsection{The developing map}

By Theorem~\ref{thm:jet-log-normal-form} we may fix, after shrinking \(B\), a
holomorphic frame \(s=(s_0,\dots,s_n)\) of \(E\) on \(B\) with
\begin{equation}\label{eq:normal-form-frame}
  \mathcal D s=s\,(1-\beta)E_{11}\frac{dz^1}{z^1} .
\end{equation}
On a simply connected sector of \(B\setminus D\) choose a branch of
\((z^1)^{\beta-1}\) and set
\begin{equation}\label{eq:parallel-frame}
  \widehat s=\bigl(s_0,\ (z^1)^{\beta-1}s_1,\ s_2,\dots,s_n\bigr),
\end{equation}
which by \eqref{eq:normal-form-frame} is a \(\mathcal D\)-parallel frame.

Dualising the first jet exact sequence gives a holomorphic line subbundle
\(\mathcal O_B\subset E^*\), represented by the evaluation functional
\(\mathrm{ev}(j(f))=f\).  This is a nowhere vanishing holomorphic section of
\(E^*\) over \(B\).  Let
\begin{equation}\label{eq:u-coordinates}
  u=(u_0,u_1,\dots,u_n),
  \qquad
  u_i=\mathrm{ev}(s_i),
\end{equation}
be its coordinates in the dual frame \(s^*\); they are holomorphic on
\(B\) and have no common zero.  In the frame \(e\) the corresponding coordinates
are \((1,z^1,\dots,z^n)\).

\begin{lemma}\label{lem:parallel-frame}
The matrix of \(\mathcal H\) in the frame \eqref{eq:parallel-frame} is constant, and
its entries \((1,j)\) with \(j\neq1\) vanish.  Moreover \(u_1\) vanishes along \(D\).
\end{lemma}

\begin{proof}
Since \(\mathcal D\widehat s=0\) and \(\mathcal D\) is the Chern connection of
\(\mathcal H\), the functions \(\mathcal H(\widehat s_i,\widehat s_j)\) have vanishing
differential on the sector, so they are constants.  For \(j\neq1\) this reads
\[
  \mathcal H(s_1,s_j)=c\,(z^1)^{1-\beta}
\]
for a constant \(c\), on the chosen branch.  The left-hand side is single-valued on
\(B\setminus D\), while continuation once around \(D\) multiplies the right-hand side
by \(e^{2\pi\ii(1-\beta)}\neq1\); hence \(c=0\).

For the last assertion, let \(\mathcal R=\Res_D(\mathcal D)\), so that the residue of
the dual connection on \(E^*\) is \(-\mathcal R^{\mathsf T}\).  By
\eqref{eq:jet-residue-block} every row of \(\mathcal R\) vanishes except row \(1\);
computing in the frame \(e\), where \(\mathrm{ev}\) has coordinates
\((1,z^1,\dots,z^n)\),
\[
  \bigl(\mathcal R^{\mathsf T}\mathrm{ev}\bigr)_j
  =\sum_i\mathcal R_{ij}\,\mathrm{ev}_i
  =z^1\,\mathcal R_{1j},
\]
which vanishes on \(D\).  Thus \(\mathrm{ev}|_D\) lies in the kernel of the residue of
the dual connection.  In the frame \(s\) that residue is
\(-(1-\beta)E_{11}\) by \eqref{eq:normal-form-frame}, and its kernel is
\(\{u_1=0\}\).
\end{proof}

Write
\begin{equation}\label{eq:u1-factor}
  u_1=z^1v
\end{equation}
with \(v\) holomorphic on \(B\).  By Lemma~\ref{lem:parallel-frame} the matrix of
\(\mathcal H\) in the frame \eqref{eq:parallel-frame}, and hence that of the induced
form \(\mathcal H^\vee\) in the coframe \(\widehat s^*\), is constant and splits the
slot \(1\) off from the remaining ones.  The restriction of \(\mathcal H^\vee\) to
the evaluation line is \(h_\sigma^{-1}>0\)
\cite[Corollary~5.2]{deBorbon2025}, so the vector
\(u(0)=(u_0(0),0,u_2(0),\dots,u_n(0))\) is positive.  Replacing \(\widehat s\) by
\(\widehat s\,A\) for a constant \(A\) respecting the splitting, we may therefore assume
that the matrix of \(\mathcal H^\vee\) in the coframe \(\widehat s^*\) is the identity
when \(\sigma=1\) and \(\operatorname{diag}(1,-1,\dots,-1)\) when \(\sigma=-1\), and,
using that the isometry group of the second block acts transitively on the positive
vectors of a given norm, that in addition
\begin{equation}\label{eq:normalisation}
  u_a(0)=0\quad(a\ge2),
  \qquad\text{so that}\qquad
  u_0(0)\neq0 .
\end{equation}

The coframe \(\widehat s^*\) now identifies \(\mathbb P(E^*)\), respectively its
domain of positive lines, with \(X_\sigma\) fibrewise isometrically, and by
\eqref{eq:parallel-frame}, \eqref{eq:u-coordinates} and \eqref{eq:u1-factor} the
section defined by the evaluation line becomes the map
\begin{equation}\label{eq:developing-map}
  F=\bigl[\,u_0:(z^1)^\beta v:u_2:\dots:u_n\,\bigr] .
\end{equation}
By \cite[Proposition~5.3]{deBorbon2025},
\begin{equation}\label{eq:jet-pullback-identity}
  F^*\omega_\sigma=\omega .
\end{equation}
Substituting \eqref{eq:developing-map} into \eqref{eq:space-form}, and discarding the
pluriharmonic term \(\log|u_0|^2\), this reads
\begin{equation}\label{eq:explicit-potential}
  \omega
  =\sigma\,\ii\,\partial\bar\partial
   \log\Bigl(1+\sigma\Bigl(|z^1|^{2\beta}|\chi|^2
     +\sum_{a=2}^n\bigl|u_a/u_0\bigr|^2\Bigr)\Bigr),
  \qquad
  \chi=v/u_0,
\end{equation}
on a neighbourhood of the origin.

\begin{lemma}\label{lem:developing-coordinates}
We have \(\chi(0)\neq0\), and for any branch of \(\chi^{1/\beta}\) the functions
\begin{equation}\label{eq:new-coordinates}
  w^1=z^1\chi^{1/\beta},
  \qquad
  w^a=u_a/u_0\quad(a\ge2)
\end{equation}
are holomorphic coordinates centred at the origin, with \(D=\{w^1=0\}\).
\end{lemma}

\begin{proof}
On the sector let \(P=\bigl((z^1)^\beta\chi,\,u_2/u_0,\dots,u_n/u_0\bigr)\) be the
map into the affine chart of \eqref{eq:space-form}, so that
\(P^*\omega_\sigma=\omega\) by \eqref{eq:explicit-potential}.  By
\eqref{eq:normalisation}, after shrinking \(B\) the image of \(P\) lies in a fixed
compact subset of the chart, on which \(g_\sigma\) is uniformly equivalent to the
Euclidean metric; hence \(\det(P^*g_\sigma)\) and \(|\det DP|^2\) are comparable.
Since \(g\) is uniformly equivalent to \(g_\beta\), whose determinant is
\(\rho^{2\beta-2}\), this gives
\[
  C^{-1}\rho^{\beta-1}\le |\det DP|\le C\rho^{\beta-1}
\]
for some constant \(C\ge1\).
On the other hand \(\partial_1P^1=\beta(z^1)^{\beta-1}\chi+O(\rho^\beta)\) while all
the other entries of \(DP\) are \(O(1)\), those in the first row being
\(O(\rho^{\beta})\).  Expanding the determinant along the first row,
\[
  \det DP=\beta(z^1)^{\beta-1}\chi\cdot
   \det\bigl(\partial_b(u_a/u_0)\bigr)_{a,b\ge2}+O(\rho^\beta),
\]
so the displayed lower bound forces \(\chi(0)\neq0\) and
\(\det\bigl(\partial_b(u_a/u_0)\bigr)_{a,b\ge2}\neq0\) at the origin.  The functions
\eqref{eq:new-coordinates} then vanish at the origin by \eqref{eq:normalisation}, and
their Jacobian there is block triangular, since \(\partial_aw^1\) vanishes on
\(D\), with determinant
\(\chi(0)^{1/\beta}\det\bigl(\partial_b(u_a/u_0)\bigr)_{a,b\ge2}\neq0\).  Finally
\(w^1\) and \(z^1\) differ by the nonvanishing factor \(\chi^{1/\beta}\), so
\(D=\{w^1=0\}\).
\end{proof}

\subsection{Proof of the classification in nonzero curvature}

\begin{proof}[Proof of Theorem~\ref{thm:main}(ii)]
Let \(w\) be the coordinates of Lemma~\ref{lem:developing-coordinates}.  Then
\(|w^1|^{2\beta}=|z^1|^{2\beta}|\chi|^2\), so \eqref{eq:explicit-potential} becomes
\[
  \omega=\sigma\,\ii\,\partial\bar\partial
   \log\Bigl(1+\sigma\bigl(|w^1|^{2\beta}
     +\textstyle\sum_{a=2}^n|w^a|^2\bigr)\Bigr),
\]
which proves the result in the normalisation \(c=2\sigma\).  Multiplying this
formula by \(2/|c|\) gives \eqref{eq:main-chsc} for the original metric.
Equivalently, by
\eqref{eq:developing-map} and \eqref{eq:new-coordinates} the developing map is
\(F(w)=\bigl((w^1)^\beta,w'\bigr)\) in the affine chart, and
\eqref{eq:jet-pullback-identity} is the assertion of the theorem.
\end{proof}

\begin{remark}
The nonzero-curvature argument uses two structures on \(E\).  The flat Chern
connection \(\mathcal D\) determines the residue \(1-\beta\), while the evaluation
line \(\mathcal O_B\subset E^*\), obtained by dualising the first jet exact sequence,
determines the developing map.  The passage from
\eqref{eq:developing-map} to \eqref{eq:new-coordinates} is the higher-dimensional
form of the normalisation used in \cite{deBorbonPanov2022} to identify the local
model of a spherical metric at a cone point.
\end{remark}

\clearpage
\appendix

\section{Polyhomogeneity of K\"ahler--Einstein cone metrics}
\label{app:polyhomogeneity}

In this appendix we show that the regularity hypothesis in
Theorem~\ref{thm:main} follows from its other assumptions.  The argument is local
along \(D\), so throughout the proof we may work in a smaller ball compactly
contained in \(B\).

\begin{theorem}[Local polyhomogeneity]\label{thm:appendix-polyhomogeneity}
Let \(g\) be a K\"ahler metric on \(B\setminus D\), where \(D=\{z^1=0\}\), and
suppose that, for some \(0<\beta<1\) and \(C\ge1\),
\[
  C^{-1}g_\beta\le g\le Cg_\beta.
\]
Assume moreover that
\begin{equation}\label{eq:app-KE}
  \Ric(g)=\kappa g
\end{equation}
on \(B\setminus D\), for some \(\kappa\in\R\).  Then \(g\) is
polyhomogeneous along \(D\) in the sense of
Definition~\ref{def:polyhomogeneous}.
\end{theorem}

\subsection{The weak conical potential}

We first obtain from the metric comparison the weak potential required by
Chen--Wang.

Following \cite[Definition~1.3, pp.~970--971]{ChenWang2017}, a function \(u\) on
an open set \(U\) meeting \(D\) belongs to \(C^{1,1,\beta}(U)\) if, for some
\(0<\alpha<1\),
\[
  u\in C^{2,\alpha}(U\setminus D)\cap C^\alpha(U)
\]
and there is a constant \(K>0\) such that
\[
  -K\omega_\beta\le \ii\,\partial\bar\partial u\le K\omega_\beta
  \qquad\text{on }U\setminus D.
\]
A closed positive \((1,1)\)-current \(\eta\) on \(U\) is a \emph{weak conical
K\"ahler metric} if it admits locally a plurisubharmonic
\(C^{1,1,\beta}\)-potential and, for some \(K>1\),
\[
  K^{-1}\omega_\beta\le \eta\le K\omega_\beta
  \qquad\text{on }U\setminus D.
\]

\begin{proposition}[Weak potential]\label{prop:app-weak-potential}
Let \(g\) be a K\"ahler metric on \(B\setminus D\) satisfying
\(C^{-1}g_\beta\le g\le Cg_\beta\), and let \(\omega\) be its K\"ahler form.
Then \(\omega\) has locally integrable coefficients near \(D\), and its trivial
extension across \(D\) is a closed positive current.  Locally along \(D\), this
current admits a plurisubharmonic potential \(\varphi\) such that
\[
  \omega=\ii\,\partial\bar\partial\varphi,
  \qquad
  \varphi\in C^\gamma
  \quad\text{for every}\quad
  0<\gamma<\min\{2\beta,1\}.
\]
Moreover,
\begin{equation}\label{eq:app-hessian-comparison}
  C^{-1}\omega_\beta
  \le \ii\,\partial\bar\partial\varphi
  \le C\omega_\beta
\end{equation}
on \(B\setminus D\).  In particular, \(\omega\) is a weak conical K\"ahler
metric in the sense of \cite[Definition~1.3]{ChenWang2017}.
\end{proposition}

\begin{proof}
Put \(\rho=|z^1|\).  The metric comparison and positivity give
\[
  g_{1\bar1}=O(\rho^{2\beta-2}),
  \qquad
  g_{a\bar b}=O(1),
  \qquad
  |g_{1\bar a}|^2\le g_{1\bar1}g_{a\bar a}
    =O(\rho^{2\beta-2}).
\]
Consequently
\[
  \tr_{\omega_{\mathrm{Euc}}}\omega
  \le C(1+\rho^{2\beta-2}).
\]
Since
\[
  \int_0^s\rho^{2\beta-2}\rho\,d\rho=O(s^{2\beta}),
\]
the form \(\omega\) has locally finite mass near \(D\).  Its trivial extension
is positive, and, since \(D\) is complete pluripolar, the Skoda--El Mir
extension theorem implies that it is closed; see
\cite[Chapter~III, Section~2.A, Theorem~2.3]{Demailly2012}.  The local
\(\partial\bar\partial\)-lemma for currents then gives a plurisubharmonic
\(\varphi\in L^1_{\mathrm{loc}}\) with
\(\omega=\ii\,\partial\bar\partial\varphi\).  The estimate
\eqref{eq:app-hessian-comparison} is the original metric comparison.

It remains to establish the H\"older regularity.  Taking a Euclidean trace gives,
in the sense of distributions,
\[
  \Delta_{\mathrm{Euc}}\varphi=f,
  \qquad
  0\le f\le C(1+\rho^{2\beta-2}).
\]
In real dimension \(2n\), integration first in the two normal variables gives
\begin{equation}\label{eq:app-mass-growth}
  \int_{B_s(x)}f\le Cs^{2n-2+2\beta}
\end{equation}
for every sufficiently small Euclidean ball.  If the ball stays away from \(D\),
the same estimate follows directly from the pointwise bound for \(f\).

Fix \(0<\gamma<\min\{2\beta,1\}\) and choose
\(\gamma<\delta<\min\{2\beta,1\}\).  By
\eqref{eq:app-mass-growth},
\(\int_{B_s(x)}f\le Cs^{2n-2+\delta}\).  Let \(\Phi\) be the Newton potential
of a compactly supported cut-off of \(f\), equal to \(f\) on a smaller ball,
and put \(d=|x-y|\).  First suppose that \(n\ge2\).  The region within distance
\(2d\) of \(x\) or \(y\) contributes at most \(Cd^\delta\).  Decompose the
remaining region into the annuli
\(A_k=\{2^kd\le |x-z|<2^{k+1}d\}\), \(k\ge1\).  On \(A_k\), the difference of
the Newton kernels is at most \(Cd(2^kd)^{1-2n}\), while the mass is at most
\(C(2^kd)^{2n-2+\delta}\).  Their sum is therefore bounded by
\(Cd^\delta\sum_{k\ge1}2^{-k(1-\delta)}\).  When \(n=1\), the Newton kernel is
logarithmic.  The annuli give the same estimate, while the region near \(x\)
and \(y\) contributes at most \(Cd^\delta(1+|\log d|)\).  Since
\(\gamma<\delta\), in both cases
\(|\Phi(x)-\Phi(y)|\le C|x-y|^\gamma\).
On the smaller ball, \(\varphi-\Phi\) is distributionally harmonic, so Weyl's
lemma shows that it is smooth.  Thus \(\varphi\in C^\gamma\).  Away from \(D\),
the potential is smooth because its complex Hessian is \(\omega\), and the
proposition follows.
\end{proof}

We shall also need a first-derivative bound in order to verify the precise
hypothesis on the right-hand side of the Chen--Wang equation.  On a simply
connected sector choose the cone coordinate
\[
  \zeta=\frac{(z^1)^\beta}{\beta},
  \qquad r=|\zeta|.
\]

\begin{lemma}[Gradient bound]\label{lem:app-gradient}
Choose
\[
  \beta<\gamma<\min\{2\beta,1\}.
\]
Then
\[
  |\partial_{z'}\varphi|\le C,
  \qquad
  |\partial_\zeta\varphi|\le Cr^{\gamma/\beta-1}.
\]
In particular,
\begin{equation}\label{eq:app-gradient-bounded}
  |\partial\varphi|_{g_\beta}\le C.
\end{equation}
\end{lemma}

\begin{proof}
In \((\zeta,z')\), the complex Hessian of \(\varphi\) is uniformly bounded by
\eqref{eq:app-hessian-comparison}.  On a tangential slice, the restriction of
\(\varphi\) has bounded Laplacian and bounded oscillation on a ball of fixed
radius.  The interior gradient estimate for the Poisson equation therefore gives
\(|\partial_{z'}\varphi|\le C\).

Fix a point of \(B\setminus D\) and put \(r=|\zeta|\) there.  On the normal
slice, a disc of radius comparable to \(r\) has diameter comparable to
\(\rho=|z^1|\) in the original coordinate.  Proposition
\ref{prop:app-weak-potential} gives
\[
  \operatorname{osc}\varphi=O(\rho^\gamma)=O(r^{\gamma/\beta}).
\]
The same interior gradient estimate now yields
\[
  |\partial_\zeta\varphi|
  \le C\bigl(r^{\gamma/\beta-1}+r\bigr)
  \le Cr^{\gamma/\beta-1},
\]
where \(0<\gamma/\beta-1<1\).  This proves the lemma.
\end{proof}

\subsection{Chen--Wang regularity}

\begin{proposition}[Conical H\"older regularity]\label{prop:app-CW}
Under the hypotheses of Theorem~\ref{thm:appendix-polyhomogeneity}, the
potential \(\varphi\) belongs locally to \(C^{2,\alpha,\beta}\) for every
\begin{equation}\label{eq:app-alpha-range}
  0<\alpha<\min\left\{\frac1\beta-1,1\right\}.
\end{equation}
\end{proposition}

\begin{proof}
Set
\[
  h=\log\frac{\omega^n}{\omega_\beta^n}.
\]
The metric comparison makes \(h\) bounded.  Since \(\omega_\beta\) is
Ricci-flat away from \(D\), equation \eqref{eq:app-KE} gives
\[
  \ii\,\partial\bar\partial(h+\kappa\varphi)=0
\]
on \(B\setminus D\).  Thus \(H=h+\kappa\varphi\) is bounded and
pluriharmonic there.

The function \(H\) extends pluriharmonically across \(D\).  Indeed, on each
normal punctured disc it is bounded and harmonic, so the puncture is removable.
On a smaller product neighbourhood the extension is given by the Poisson formula
from a fixed normal circle.  Its boundary values are smooth in the tangential
variables, hence the extension is smooth jointly in \((z^1,z')\).  Since
\(\partial\bar\partial H=0\) off \(D\), it remains pluriharmonic across \(D\).
Both sides of the following identity put no mass on \(D\): for the left-hand
side this follows from the local boundedness of the plurisubharmonic potential,
and for the right-hand side it follows from the local integrability of
\(\omega_\beta^n\).  Consequently
\begin{equation}\label{eq:app-MA}
  \omega^n=e^{H-\kappa\varphi}\omega_\beta^n
\end{equation}
in the sense of measures.

Put \(F=e^{H-\kappa\varphi}\) and \(\Theta=H-\kappa\varphi\).  The function
\(F\) is positive, smooth away from \(D\), and H\"older continuous across it.
Since \(H\) is smooth, Lemma~\ref{lem:app-gradient} gives
\(|\partial\Theta|_{g_\beta}\le C\), while
\[
  \ii\,\partial\bar\partial\Theta=-\kappa\omega.
\]
It follows that
\[
  \ii\,\partial\bar\partial F
  =F\left(-\kappa\omega
    +\ii\,\partial\Theta\wedge\bar\partial\Theta\right),
\]
and hence
\[
  -C\omega_\beta\le
  \ii\,\partial\bar\partial F
  \le C\omega_\beta.
\]
Thus \(F\in C^{1,1,\beta}\) in the sense of
\cite[Definition~1.3]{ChenWang2017}.  Proposition
\ref{prop:app-weak-potential} gives the same membership for \(\varphi\), and
\(\ii\,\partial\bar\partial\varphi\) is a weak conical metric.  Applying
\cite[Theorem~1.6]{ChenWang2017} to \eqref{eq:app-MA} proves
the proposition for every \(\alpha\) satisfying \eqref{eq:app-alpha-range}.
Notice that this formulation applies to either sign of \(\kappa\).
\end{proof}

\subsection{The polyhomogeneous expansion}

In holomorphic coordinates, equation \eqref{eq:app-MA} becomes, up to an
inessential positive constant depending only on the normalisation of
\(\omega_\beta\),
\begin{equation}\label{eq:app-YZ-equation}
  \det(\varphi_{i\bar j})
  =\frac{e^{-\kappa\varphi+H}}{|z^1|^{2-2\beta}}.
\end{equation}
The function \(H\) is smooth, Proposition~\ref{prop:app-CW} gives
\(\varphi\in C^{2,\alpha}_\beta\) in Donaldson's notation
\cite{Donaldson2012}, and the metric comparison is unchanged.  The hypotheses of
\cite[Theorem~1.1]{YinZheng2019} therefore hold with
\(\lambda=-\kappa\) and smooth term \(H\).

Yin--Zheng prove that, to every finite order, \(\varphi\) is a finite linear
combination, with coefficients depending smoothly on \(z'\), of
\begin{equation}\label{eq:app-YZ-terms}
  r^{2j+k/\beta}(\log r)^m\cos(l\theta),
  \qquad
  r^{2j+k/\beta}(\log r)^m\sin(l\theta),
\end{equation}
where
\[
  j,k,l,m\in\N\cup\{0\},
  \qquad
  \frac{k-l}{2}\in\N\cup\{0\},
  \qquad
  m\le\max\{0,k-1\}.
\]
Although \cite[Theorem~1.1]{YinZheng2019} is stated for each fixed tangential
parameter, the estimates in its proof are locally uniform in that parameter.  In
particular, \cite[Theorem~3.1]{YinZheng2019} gives uniform bounds for arbitrary
tangential derivatives on a smaller product neighbourhood, and the subsequent
expansion argument uses these estimates.  Hence the conormal remainder estimates
are locally uniform in \(z'\), and the same expansions hold after arbitrary
tangential differentiation.

For completeness, we relate this conclusion precisely to
Definition~\ref{def:polyhomogeneous}.  Below any fixed order there are only
finitely many exponents \(2j+k/\beta\).  At a fixed exponent, Fourier projection
in \(\theta\) and recursive extraction of the logarithmic coefficients makes the
coefficients unique.  Applying the same extraction after arbitrary tangential
differentiation shows that these coefficients are smooth functions of \(z'\) and
that tangential differentiation may be performed term by term.  Together with
the conormal remainder estimates of \cite[Theorems~1.1 and~3.1]{YinZheng2019},
this is the expansion of Definition~\ref{def:phg}.

We now verify the two additional requirements in
Definition~\ref{def:polyhomogeneous}.  The terms with \(j=k=0\) have
\(l=m=0\) and combine to give \(\varphi_0(z')\).  Every other exponent satisfies
\[
  2j+\frac{k}{\beta}\ge
  \begin{cases}
    2,&j\ge1,\\
    1/\beta,&j=0,\ k\ge1,
  \end{cases}
\]
and is therefore strictly greater than \(1\).  This gives
\eqref{eq:polyhomogeneous} with \(\sigma_1>1\).

It remains to consider the limit of the metric matrix \(G\) in cone coordinates.
The only terms of order less than \(2\) that can occur are, when
\(1/\beta<2\), linear combinations of
\[
  r^{1/\beta}\cos\theta,
  \qquad
  r^{1/\beta}\sin\theta.
\]
On a sector these are constant multiples of the real and imaginary parts of
\(z^1=(\beta\zeta)^{1/\beta}\).  Their normal complex Hessian vanishes, their
mixed derivatives are \(O(r^{1/\beta-1})\), and their tangential second
derivatives are \(O(r^{1/\beta})\); hence none contributes to the limit of
\(G\).  The same conclusion holds for the possible order-two angular term when
\(\beta=1/2\).

The non-angular term of order \(2\) is
\(a(z')r^2=a(z')|\zeta|^2\); its normal complex Hessian has the
\(\theta\)-independent limit \(a(z')\), while its mixed and tangential components
tend to zero.  The term \(\varphi_0(z')\) supplies the limiting tangential block.
Every remaining term has order strictly greater than \(2\), so after two normal
derivatives it is \(o(1)\), including all of its logarithmic factors.  The
conormal remainder estimates make these conclusions uniform in \(\theta\).
Therefore
\[
  G\longrightarrow G_0(z')
  \qquad\text{as }r\to0,
\]
which is \eqref{eq:G-limit}.  Both additional requirements of
Definition~\ref{def:polyhomogeneous} are now verified, and the proof of
Theorem~\ref{thm:appendix-polyhomogeneity} is complete.

\bibliographystyle{alpha}
\bibliography{refs}

\begin{thebibliography}{JMR16}

\bibitem[Boc47]{Bochner1947}
Salomon Bochner.
\newblock Curvature in {H}ermitian metric.
\newblock {\em Bulletin of the American Mathematical Society}, 53:179--195,
  1947.

\bibitem[CW17]{ChenWang2017}
Xiuxiong Chen and Yuanqi Wang.
\newblock On the regularity problem of complex {Monge--Amp\`ere} equations with
  conical singularities.
\newblock {\em Annales de l'Institut Fourier}, 67(3):969--1003, 2017.

\bibitem[dB25]{deBorbon2025}
Martin de~Borbon.
\newblock Local classification of {K\"ahler} metrics with constant holomorphic
  sectional curvature.
\newblock {\em The Quarterly Journal of Mathematics}, 76(4):1387--1398, 2025.

\bibitem[dBP21]{deBorbonPanov2021}
Martin de~Borbon and Dmitri Panov.
\newblock Polyhedral {K}\"ahler cone metrics on {$\mathbb{C}^n$} singular at
  hyperplane arrangements.
\newblock {\em arXiv preprint}, arXiv:2106.13224, 2021.

\bibitem[dBP22]{deBorbonPanov2022}
Martin de~Borbon and Dmitri Panov.
\newblock Parabolic bundles and spherical metrics.
\newblock {\em Proceedings of the American Mathematical Society},
  150(12):5459--5472, 2022.

\bibitem[Del70]{Deligne1970}
Pierre Deligne.
\newblock {\em \'{E}quations diff\'{e}rentielles \`a points singuliers
  r\'{e}guliers}, volume 163 of {\em Lecture Notes in Mathematics}.
\newblock Springer-Verlag, Berlin, 1970.

\bibitem[Dem12]{Demailly2012}
Jean-Pierre Demailly.
\newblock {\em Complex Analytic and Differential Geometry}.
\newblock Institut Fourier, Universit\'{e} Grenoble Alpes, 2012.
\newblock OpenContent book.

\bibitem[DG23]{DeroinGuillot2023}
Bertrand Deroin and Adolfo Guillot.
\newblock Foliated affine and projective structures.
\newblock {\em Compositio Mathematica}, 159(6):1153--1187, 2023.

\bibitem[Don12]{Donaldson2012}
Simon~K. Donaldson.
\newblock K\"ahler metrics with cone singularities along a divisor.
\newblock In {\em Essays in Mathematics and its Applications}, pages 49--79.
  Springer, Heidelberg, 2012.

\bibitem[JMR16]{JMR2016}
Thalia~D. Jeffres, Rafe Mazzeo, and Yanir~A. Rubinstein.
\newblock K\"ahler--{E}instein metrics with edge singularities.
\newblock {\em Annals of Mathematics}, 183(1):95--176, 2016.

\bibitem[KZ18]{KellerZheng2018}
Julien Keller and Kai Zheng.
\newblock Construction of constant scalar curvature {K}\"ahler cone metrics.
\newblock {\em Proceedings of the London Mathematical Society (3)},
  117(3):527--573, 2018.

\bibitem[NY04]{NovikovYakovenko2004}
Dmitry Novikov and Sergei Yakovenko.
\newblock Lectures on meromorphic flat connections.
\newblock In {\em Normal Forms, Bifurcations and Finiteness Problems in
  Differential Equations}, volume 137 of {\em NATO Science Series II:
  Mathematics, Physics and Chemistry}, pages 387--430. Kluwer Academic
  Publishers, Dordrecht, 2004.

\bibitem[SW16]{SongWang2016}
Jian Song and Xiaowei Wang.
\newblock The greatest {R}icci lower bound, conical {E}instein metrics and
  {C}hern number inequality.
\newblock {\em Geometry \& Topology}, 20(1):49--102, 2016.

\bibitem[Tro86]{Troyanov1986}
Marc Troyanov.
\newblock Les surfaces euclidiennes \`a singularit\'{e}s coniques.
\newblock {\em L'Enseignement Math\'{e}matique}, 32(1--2):79--94, 1986.

\bibitem[YZ19]{YinZheng2019}
Hao Yin and Kai Zheng.
\newblock Expansion formula for complex {Monge--Amp\`ere} equation along cone
  singularities.
\newblock {\em Calculus of Variations and Partial Differential Equations},
  58(2):Paper No. 50, 32 pp., 2019.

\end{thebibliography}

\bigskip
\noindent\textsc{Loughborough University, Loughborough LE11 3TU, UK}\par
\noindent\textit{Email address}: \texttt{m.de-borbon@lboro.ac.uk}

\end{document}